\documentclass{amsart}
\usepackage{graphicx}
\usepackage{amssymb,amscd,amsthm,amsxtra}
\usepackage{latexsym}
\usepackage{epsfig}

\newtheorem{thm}{Theorem}[section]

\newtheorem{lem}[thm]{Lemma}
\newtheorem{prop}[thm]{Proposition}
\theoremstyle{definition}
\newtheorem{defn}[thm]{Definition}
\theoremstyle{remark}
\newtheorem{rem}[thm]{Remark}
\numberwithin{equation}{section}

\newcommand{\R}{\mathbb R}

\newcommand{\p}{\partial}

\newcommand{\comment}[1]{}

\begin{document}

\title[Boundary Harnack in slit domains]{Boundary Harnack estimates in slit domains and applications to thin free boundary problems}
\author{D. De Silva}
\address{Department of Mathematics, Barnard College, Columbia University, New York, NY 10027}
\email{\tt  desilva@math.columbia.edu}
\author{O. Savin}
\address{Department of Mathematics, Columbia University, New York, NY 10027}\email{\tt  savin@math.columbia.edu}

\thanks{ D.~D. is supported by NSF grant DMS-1301535. 
O.~S.~ is supported by NSF grant DMS-1200701.}

\begin{abstract}
We provide a higher order boundary Harnack inequality for harmonic functions in slit domains. As a corollary we obtain the $C^\infty$ regularity of the free boundary in the Signorini problem near non-degenerate points. 
\end{abstract}
\maketitle

\section{Introduction}

In our recent work \cite{DS4} we proved a higher order boundary Harnack estimate for harmonic functions vanishing on a part of the boundary of a domain $\Omega$ in $\R^n$. 
We recall briefly the main result. 

\begin{thm}\label{BH} Let $v$ and $u>0$ be two harmonic functions in $\Omega \subset \R^n$ that vanish continuously on some portion of the boundary $\Gamma \subset \p \Omega$ and let $k \ge 1$.

If $\Gamma \in C^{k,\alpha}$ then $\frac vu \in C^{k,\alpha}$ up to the boundary in a neighborhood of $\Gamma$.
\end{thm}

By classical Schauder estimates both functions $u$ and $v$ are of class $C^{k,\alpha}$ up to the boundary. In general, the quotient of two $C^{k,\alpha}$ functions that vanish on the boundary is only of class $C^{k-1,\alpha}$ in a neighborhood of the boundary. The theorem states that the quotient of two harmonic functions is in fact one derivative better than what is expected from boundary Schauder estimates.

Theorem \ref{BH} is well known as the {\it boundary Harnack theorem} in the case $k=0$ (see \cite{HW}, \cite{CFMS}, \cite{Fe}). An easy application of Theorem \ref{BH} gives $C^\infty$ regularity for $C^{1,\alpha}$ free boundaries in the classical obstacle problem, see \cite{DS4}. 

In this paper we obtain the corresponding theorems in the case of slit domains and the thin obstacle problem.
A {\it slit domain} is a domain in $\R^{n+1}$ from which we remove an $n$-dimensional set $\mathcal P \subset \{x_{n+1}=0\}$ (slit), with $C^{k,\alpha}$ boundary in $\R^n$, $\Gamma:=\p_{\R^n} \mathcal P$, $k \ge 1$.

In \cite{DS3} we investigated the higher regularity of the free boundary for the {\it thin one-phase} free boundary problem (see \cite{CRS}, \cite{DR}, \cite{DS1}, \cite{DS2}). In particular we developed a Schauder type estimate in slit domains, see Theorem \ref{Schauder} in the next section for the precise statement. The Schauder estimate states that if $u$ is even in $x_{n+1}$, and it is harmonic in the slit domain $B_1 \setminus \mathcal P$, and $u$ vanishes continuously on the slit $\mathcal P$ then
$$\Gamma \in C^{k,\alpha} \quad \Longrightarrow \quad \frac{u}{U_0} \in C^{k-1,\alpha}(x_1,...,x_n,r), \quad \quad \quad \quad k \ge 2.$$
Here $$U_0:=\frac{1}{\sqrt 2} \sqrt{d+r}$$ where $d$ represents the signed distance to $\Gamma$ in $\R^n$ and $r$ denotes the distance to $\Gamma$ in $\R^{n+1}$.  
The above statement says that the quotient $u/U_0$ is essentially a $C^{k-1,\alpha}$ function in the variables $(x_1,..,x_n,r)$. 

The explicit function $U_0$ is an approximation of a harmonic function which vanishes on $\mathcal P$ and plays the same role as the distance function in smooth (non-slit) domains. 
For example when $\Gamma=\{(x_n,x_{n+1})=(0,0)\}$ is a ``straight" boundary then $U_0$ is indeed harmonic.

In this paper we obtain the boundary Harnack estimate in slit domains, see Theorem \ref{SchauderO} in the next section. We show that if $u$ and $U>0$ are even, harmonic functions vanishing on $\mathcal P$ then
$$\Gamma \in C^{k,\alpha} \quad \Longrightarrow \quad \frac{u}{U} \in C^{k,\alpha}(x_1,...,x_n,r), \quad \quad \quad \quad k \ge 1.$$
 We also provide the Schauder estimates in slit domains with $C^{1,\alpha}$ boundary, which were not completed in \cite{DS2}.

The proofs of the Schauder estimates and the boundary Harnack estimates are essentially the same. We approximate $u$ in a sequence of concentric balls $B_{\rho^l}$ by functions $U_0 P$ or $U\, P$ with $P(x,r)$ a polynomial in $x$ and $r$. The gain of one extra derivative comes from the fact that while $U$ is harmonic, the explicit function $U_0$ only approximates a harmonic function up to an error, thus $U \, P$ provides a better approximation than $U_0 P$.
 
 \
 
 {\bf Signorini problem.}
 As an application of the boundary Harnack estimates in slit domains we obtain $C^\infty$ regularity of the free boundary near regular points in the Signorini problem, also known as the thin obstacle problem.  It consists in minimizing
 \begin{equation}\label{SN}
  \min_{u \in \mathcal A} \, \, \, \int_{B_1} |\nabla u|^2 dX ,  
  \end{equation}
 with $B_1 \subset \R^{n+1}$, $X=(x,x_{n+1}) \in \R^{n+1}, $ and $\mathcal A$ the convex set
 $$\mathcal A:=\left\{u \in H^1(B_1), \quad u=\varphi \quad \mbox{on $\p B_1$}, \quad u(x,0) \ge 0  \right \}.$$ 
 
 There is considerable literature on the regularity properties of the solution
(see \cite{F}, \cite{C}, \cite{U}). In particular, the minimizer $u$ is Lipschitz in $B_1$, and is harmonic in the slit domain $B_1 \setminus \mathcal P$ with $\mathcal P:=\{ u=0\} \cap \{x_{n+1}=0\}$. Athanasopoulos and Caffarelli obtained in \cite{AC} the optimal regularity of the solution on the free boundary $\Gamma:=\p_{\R^n} \mathcal P  \cap B_1$ i.e. the solution $u$ is pointwise $C^{1, \frac 1 2}$ at all points on $\Gamma$.
 A point $X_0 \in \Gamma$ is called a {\it singular point} of the free boundary if 
 $$u(X)=o(|X-X_0|^\frac 32). $$
 Otherwise, $X_0$ is called a {\it regular point} of the free boundary.

Concerning the regularity of the free boundary $\Gamma$, Athanasopoulos, Caffarelli and Salsa showed in \cite{ACS} that if $X_0$ is a regular point, then $\Gamma$ is given locally by the graph of a $C^{1,\alpha}$ function $g$ in some direction, say the $e_n$ direction (see also \cite{GP} concerning the singular set). Moreover, the derivatives $u_i$, $i=1,2,..,n$, are harmonic in the slit domain $B_1 \setminus \mathcal P$ and vanish continuously on $\mathcal P$, and $u_n>0$ in a neighborhood of $X_0$. 
 We remark that we may assume that $u$ is even in the $x_{n+1}$ variable since the even part of $u$ solves the obstacle problem with the same free boundary.

Now we can apply the boundary Harnack estimate, Theorem \ref{SchauderO}, and improve the $C^{1,\alpha}$ regularity of $\Gamma$ to $\Gamma \in C^\infty$. Indeed, since $\Gamma \in C^{1,\alpha}$, we obtain that $\frac{u_i}{u_n}$ is $C^{1,\alpha}$ when restricted to $\Gamma$ in a neighborhood of $X_0$. On the other hand $\frac{u_i}{u_n}$ restricted to $\Gamma$ represents the derivative $g_i$. In conclusion $g \in C^{2,\alpha}$, hence $\Gamma \in C^{2,\alpha}$. We iterate this indefinitely and obtain $\Gamma \in C^\infty$.
 
 \begin{thm}
 Let $X_0$ be a regular point of the free boundary $\Gamma$ of a solution $u$ to the Signorini problem \eqref{SN}. Then $\Gamma \in C^\infty$ in a neighborhood of $X_0$.
 \end{thm}
  
 We remark that analyticity of the free boundary $\Gamma$ near regular points was obtained by Koch, Petrosyan and Shi in \cite{KPS} at the same time this paper was completed. They used a different method based on partial Legendre transformation.

The paper is organized as follows. In Section 2 we introduce some notation and state our main result, the boundary Harnack estimate Theorem \ref{SchauderO}. In Section 3 we prove this theorem in the case  $\Gamma \in C^{k,\alpha}$ with $k \ge 2$. In Section 4 we obtain both the Schauder and the boundary Harnack estimates when $\Gamma \in C^{1,\alpha}$.  Finally in Section 5 we collect the proofs of some technical lemmas used in our proofs. 

\section{Notation and Statement of main results}

\subsection{Notation}

Let $\Gamma$ be a $C^{k+1,\alpha}$ surface in $\R^n$, $k \geq 0.$ Assume for simplicity that $\Gamma$ is given by the graph of a function $g$ of $n-1$ variables
\begin{equation}\label{G}
\Gamma:=\{(x', g(x'))\}, \quad \quad \quad g:B_1' \subset \R^{n-1}\to \R,
\end{equation}
satisfying
\begin{equation*}\label{eqg}
g(0)=0, \quad \nabla_{x'} g(0) =0, \quad \|g\|_{C^{k+1,\alpha}(B_1')} \leq 1.
\end{equation*}

Let $\mathcal P$ denote the $n$ dimensional slit in $\R^{n+1}$ given by
$$\mathcal P := \{X=(x,x_{n+1}) \in B_1 \ | \ x_{n+1}=0, \quad x_n \leq g(x')\}.$$
Notice that in the $n$ dimensional ball $B_1' \times\{0\}$ we have $\p_{\R^n} \mathcal P =\Gamma$.

Given a point $X=(x,x_{n+1})$ we denote by $d$ the signed distance in $\R^n$ from $x$ to $\Gamma$ with $d>0$ above $\Gamma$ (in the $e_n$ direction). Let $$r:=\sqrt{x_{n+1}^2+d^2}$$  be the distance in $\R^{n+1}$ from $X$ to $\Gamma.$
We have
\begin{equation}\label{dr}
\nabla_x r= \frac d r \, \nu, \quad \quad \nu=\nabla_x d,
\end{equation}
where $\nu(x)$ represents the unit normal in $\R^n$ to the parallel surface to $\Gamma$ passing through $x$.

We recall the definition of the class $C^{k,\alpha}_{xr}$ introduced in \cite{DS3}.
We denote by
$$P(x,r) = \, a_{\mu m}\,  x^\mu r^m, \quad \deg P=k,$$
a polynomial of degree $k$ in the $(x,r)$ variables, and we use throughout the paper the summation convention over repeated indices.
Also
$$x^{\mu}=x_1^{\mu_1}\ldots x_{n}^{\mu_n}, \quad |\mu| = \mu_1+ \ldots +\mu_n, \quad \quad \mu_i \ge 0.$$
Sometimes we think that $a_{\mu m}$ are defined for all indices $(\mu,m)$, by extending them to be $0$.
We also denote $$\|P\|:= \max|a_{\mu m}|.$$

\begin{defn}We say that a function $f: B_1 \subset \R^{n+1} \rightarrow \R$ is { \it pointwise $C^{k,\alpha}$ in the $(x,r)$-variables at $0 \in \Gamma$} and write $f \in C_{xr}^{k,\alpha}(0)$ if there exists a (tangent) polynomial $P_0(x,r)$ of degree $k$
 such that $$f(X) = P_0(x, r) + O(|X|^{k+\alpha}).$$
\end{defn}

We define $\|f\|_{C^{k,\alpha}_{xr}(0)}$ as the smallest constant $M$ such that $$\|P_0\| \leq M, \quad \mbox{and} \quad |f(X) - P_0(x,r)| \leq M|X|^{k+\alpha},$$ for all $X$ in the domain of definition.

Similarly, we may write the definition for $f$ to be pointwise $C^{k,\alpha}_{xr}$ at some other point $Z \in \Gamma$. 
\

\begin{defn}Let $K \subset \Gamma.$ We say that $f \in C_{xr}^{k,\alpha} (K)$ if there exists a constant $M$ such that $f \in C_{xr}^{k,\alpha}(Z)$ for all $Z \in K$ and $\|f\|_{C_{xr}^{k,\alpha}(Z)} \leq M$ for all $Z \in K.$

The smallest $M$ in the definition above is denoted by $\|f\|_{C_{x,r}^{k,\alpha}(K)}.$
\end{defn}

\smallskip

We remark that if $f \in  C_{xr}^{k,\alpha} (\Gamma)$ and $\Gamma \in C^{k,\alpha}$ then the restriction of $f$ on $\Gamma$ is a $C^{k,\alpha}$ function.

Finally, let $\theta \in (-\pi, \pi]$ be the angle between the segment of length $r$ from $X$ to $\Gamma$ and  the $x$-hyperplane and define 
\begin{equation}\label{U0}U_0(X):= r^{1/2}\cos \frac{\theta}{2}= \frac{1}{\sqrt 2} \sqrt{d+r}.\end{equation}

\subsection{Main results.} Let $u \in C(B_1)$ be even in the $x_{n+1}$ coordinate, with $\|u\|_{L^\infty} \leq 1$, and
\begin{equation}\label{FB}
\begin{cases}
\Delta u = \dfrac{U_0}{r} f \quad \quad \text{in $B_1 \setminus \mathcal P$}\\
u=0 \quad \text{on $\mathcal P$.}
\end{cases}
\end{equation}
Also, let $U$ be a positive harmonic function, even in $x_{n+1}$, which is normalized such that $U( \frac 12 e_n)=1$. Precisely we assume that  $U> 0$ solves the problem above with $f=0$, i.e.
\begin{equation}\label{FBU}
\begin{cases}
\Delta U = 0 \quad \quad \text{in $B_1 \setminus \mathcal P$}\\
U=0 \quad \text{on $\mathcal P$.}
\end{cases}
\quad \quad \quad \mbox{and} \quad U(\frac 1 2 e_n)=1.
\end{equation}

Our main result reads as follows.

\begin{thm}[Boundary Harnack in slit domains]\label{SchauderO}
Let $\Gamma \in C^{k+1,\alpha}$ satisfy \eqref{G} and $u, U$ satisfy \eqref{FB}, \eqref{FBU} with $k\geq 0$ and $$f \in C^{k,\alpha}_{xr} (\Gamma \cap B_1), \quad \quad \|f\|_{C^{k,\alpha}_{xr} (\Gamma \cap B_1)} \leq 1.$$ Then,
\begin{equation}\label{SuO}\left \|\dfrac{u}{U} \right \|_{C^{k+1,\alpha}_{xr} (\Gamma \cap B_{1/2})} \leq C \end{equation} 
with $C$ depending only on $n$, $k$ and $\alpha$.
\end{thm}

The boundary Harnack Theorem \ref{SchauderO} complements the Schauder type estimates obtained in \cite{DS3} that we state below.

\begin{thm}[Schauder estimates in slit domains]\label{Schauder}
Let $\Gamma \in C^{k+1,\alpha}$ with $k \ge 1$ satisfy \eqref{G}, and assume $u$ solves \eqref{FB} with $$f \in C^{k-1,\alpha}_{xr} (\Gamma \cap B_1), \quad \quad \|f\|_{C^{k-1,\alpha}_{xr} (\Gamma \cap B_1)} \leq 1.$$ Then,
\begin{equation}\label{Su}\left \|\dfrac{u}{U_0} \right \|_{C^{k,\alpha}_{xr} (\Gamma \cap B_{1/2})} \leq C \end{equation} and
\begin{equation}\label{SDu}\left \|\dfrac{\nabla_x u}{(U_0/r)} \right \|_{C^{k,\alpha}_{xr} ( \Gamma \cap B_{1/2})} \leq C \end{equation}
with $C$ a constant depending only on $n$, $k$ and $\alpha$.
\end{thm}
 
 In Section 4 we also prove the Schauder estimates when $\Gamma \in C^{1,\alpha}$.
 
 \begin{thm}\label{S_00} 
Let $\Gamma \in C^{1,\alpha}$ satisfy \eqref{G} with $k = 0$, and assume $u$ solves \eqref{FB} with $$|f| \le r^{\alpha -3/2}.$$  
Then,
$$ \left \|\frac{u}{U_0} \right \|_{C^\alpha_{xr}(\Gamma \cap B_{1/2})} \leq C,$$
with $C$ a constant depending only on $n$ and $\alpha$.
\end{thm}

Throughout the paper we denote by $c$, $C$ various positive constants that depend only on $n$, $k$ and $\alpha$ and we refer to them as universal constants.

\section{The case $\Gamma \in C^{k+1,\alpha}$ with $k\geq 1$.}\label{s6}

In this Section we prove Theorem \ref{SchauderO} in the case $k\geq 1.$ We follow the same strategy as in the proof of the Schauder estimates from \cite{DS3}. 
The difference is that now we approximate $u$ by functions $U(X) P(x,r)$ instead of $U_0(X)P(x,r)$ as in \cite{DS3}.

It suffices to prove the following slightly stronger pointwise estimate for $\dfrac u U$.

\begin{prop}\label{PSchauder} Let $\Gamma \in C^{k+1, \alpha}$, satisfy \eqref{G} with $k\geq 1,$ and let $U$ be as in \eqref{FBU}. 
Assume that $u\in C(B_1)$ is even and vanishes on $\mathcal P$, $\|u\|_{L^\infty} \le 1$, and
\begin{equation*}
\Delta u (X)= \frac{U_0}{r} \, R(x,r) +F(X) \quad \quad \quad \quad \mbox{in $B_1 \setminus \mathcal P$},
\end{equation*}
with
$$|F(X)| \le r^{-\frac 12}|X|^{k+\alpha} \quad \quad \mbox{and $R(x,r)$ a polynomial of degree $k$ with $\|R\| \le 1$.}$$
 There exists a polynomial $P(x,r)$ of degree $k+1$ with coefficients bounded by $C$ such that
 $$\left|\frac{u}{U}-P\right| \le C |X|^{k+1+\alpha,}$$  with $C$ depending on $k$, $\alpha$, $n$.
\end{prop}

Now Theorem \ref{SchauderO} follows at the origin (and therefore at all points in $\Gamma \cap B_{1/2}$) since $$f(X)=R(x,r)+ h(X), \quad \quad \quad \deg R=k, \quad \quad  h(X)= O(|X|^{k+\alpha}),$$ and $F:=\dfrac{U_0}{r} h(X)$ satisfies the bound above.

Before we proceed with the proof of Proposition \ref{PSchauder} we first need to express $U$ in terms of $U_0$. This is given by the Schauder estimates Theorem \ref{Schauder} applied to $U$  a solution of \eqref{FBU}. Thus $U$ satisfies the following expansion at $0 \in \Gamma$
\begin{equation}\label{expU}U(X)= U_0(X) \;(P_0(x,r) + O(|X|^{k+\alpha})),\end{equation} for some polynomial $P_0(x,r)$ of degree $k$. The derivatives $U_i$ are in fact obtained by differentiating formally this expansion in the $x_i$ direction (see Subsection 5.2 in \cite{DS3}). 
It is easy to check that
\begin{equation}\label{dr1}
\nabla_x U_0 = \frac{U_0}{2r} \,  \nu.
\end{equation}
Thus, using \eqref{dr}-\eqref{dr1} we have
\begin{equation}\label{form}
\nabla_x U = \frac{U_0}{r} \left [ \frac 12 P_0 \, \nu + r  \, \p_x P_0 + (D_r P_0) \, d \, \nu + O(|X|^{k+\alpha}) \right],
\end{equation}
where $D_r P_0$ represents the formal derivative of $P_0$ with respect to the $r$ variable. Since $\nu, d \in C_x^{k,\alpha}$ we obtain
$$U_i= \dfrac{U_0}{r}(P_0^i(x,r) + O(|X|^{k+\alpha})), \quad \quad  \deg P_0^i=k,$$
for some polynomial $P_0^i$.

In the next lemma we state that this expansion holds also for the radial derivative $\p_r U:=\nabla r \cdot \nabla U$. Its proof can be found in the appendix.

\begin{lem}\label{pr}Let $\Gamma$ and $U$ be as above.
Then 
$$|\p_rU - \frac{U_0}{r} P_0^r| \leq C \frac{U_0}{r}|X|^{k +\alpha},$$ with $\deg P_0^r=k$ and $\frac{U_0}{r}P_0^r$ is obtained by formally differentiating $U_0P_0$ in the $r$-direction i.e., 
\begin{equation}\label{radial} 
\p_r U= \frac{U_0}{r} \left [\frac 1 2 P_0 + \nabla_x P_0 \cdot (d \,  \nu) + r \,  (D_r P_0) + O(|X|^{k+\alpha}) \right].
\end{equation}
\end{lem}

We also recall the following Theorem from \cite{DS3} which deals with the case when $\Gamma$ is straight.

\begin{thm} \label{SchauderL}
Assume $\Gamma=\{x_n=0\}$ and $u\in C(B_1)$ is even, $\|u\|\le 1$ and satisfies
$$\Delta u = 0 \quad \quad \text{in $B_1 \setminus \mathcal P$}, \quad \quad u=0 \quad \mbox{on $\mathcal P$.}$$
For any $m \ge 0$, there exists a polynomial $P_0(x,r)$ of degree $m$ such that $U_0 P_0$ is harmonic in $B_1 \setminus \mathcal P$ and
$$|u-U_0 P_0| \le K |X|^{m+1} U_0,$$
for some constant $K$ depending on $m$ and $n$.
\end{thm}

We now proceed to prove Proposition \ref{PSchauder}. After careful computations are carried on, the proof follows from similar arguments as in Proposition 5.1 in \cite{DS3}.

After performing an initial dilation we may assume that:
\begin{equation}\label{flat1}\|g\|_{C^{k+1,\alpha}} \leq \delta, \quad |R| \leq \delta, \quad |F| \leq \delta r^{-1/2} |X|^{k+\alpha},\end{equation}
and after multiplying $U$ by a constant (see \eqref{expU}) we also have
\begin{equation}\label{flat2}U=U_0(1+ \delta Q_0+ \delta O(|X|^{k+\alpha})), \quad \quad \deg Q_0=k, \quad \|Q_0\|\le 1,\end{equation}
with the zero-th order term of $Q_0$ being 0.
The constant $\delta$ will be made precise later.

Next we define the notion of ``approximating polynomial" which plays a crucial role in our analysis. 

Let $\kappa(x)$ be the mean curvature of the parallel surface to $\Gamma$ passing through $x$ and $\nu(x)$ denote the normal to this parallel surface. Thus,
$$\kappa(x) = -\Delta d \in C^{k-1, \alpha}_x, \quad \nu(x)= \nabla d \in C^{k,\alpha}_x.$$ Then, one easily computes that ($m \geq 0$)
\begin{equation}\label{rm}\Delta r^m = m r^{m-2} (m-\kappa \, d).\end{equation}

Now, let $|\mu|+ m \leq k+1$ and 
 let $\bar i $ denote the multi-index with $1$ on the $i$-th position and zeros elsewhere.
Using \eqref{rm}, \eqref{dr}, and the fact that $U$ is harmonic in $B_1 \setminus \mathcal P$ we obtain
\begin{align*}
\Delta(x^\mu r^{m} U) &= U \Delta(x^\mu r^m) + 2 \nabla (x^\mu r^m) \cdot \nabla U \\
\  & = U \left (r^m \mu_i (\mu_i-1)\, x^{\mu - 2\bar i} + m \, x^\mu r^{m-2} (m-\kappa \, d) + 2 m \, r^{m-1} \frac d r \nu \cdot \nabla_x x^{\mu} \right )\\
\ &+ 2 \left (r^m \nabla x^\mu \cdot \nabla_x U + m \, x^\mu r^{m-1} \p_r U\right )\\
\ & = \frac{U}{r} I + 2II.
\end{align*}

By Taylor expansion at 0, we write (see \eqref{flat1} and recall $\nabla_{x'} g (0)=0$)
\begin{equation}\label{low} \nu^i= \delta_n^i + \ldots, \quad \kappa= \kappa(0)+\ldots, \quad d=x_n+\ldots
\end{equation}
We arrange the terms in $I$ by  the degree up to order $k$ and group the remaining ones in a remainder. Precisely, 
\begin{align*} I& =  m(m+2\mu_n) \, x^\mu r^{m-1} 
+\mu_i(\mu_i-1)\,x^{\mu-2\bar i} r^{m+1}  +  b_{\sigma l}^{\mu m} x^\sigma r^l + \delta O(|X|^{k+\alpha}),
\end{align*}
with $$b_{\sigma l}^{\mu m} \ne 0 \quad \mbox{only if} \quad \quad |\mu| + m-1<|\sigma|+l \le k.$$

Notice that the monomials $b_{\sigma l}^{\mu m} x^\sigma r^l$ have strictly higher degree than the first terms and together with the remainder can be thought as lower order terms. Also they are linear combinations of coefficients of the tangent polynomials at 0 for $d \kappa (x), d \nu^i$ thus, 
\begin{equation*}|b_{\sigma l}^{\mu m}|\le C \delta.\end{equation*}
Notice that $b_{\sigma l}^{\mu m}$ vanish in the flat case $\Gamma=\{x_n=0\}$.

To estimate $II$ we use \eqref{form}-\eqref{radial} and obtain
$$II = \frac{U_0}{r}\left [\frac 1 2 r^m \mu_n x^{\mu - \bar n} + \frac 1 2 m x^\mu r^{m-1} + p_{\sigma l}^{\mu m } x^\sigma r^l + \delta O(|X|^{k+\alpha})\right],$$
where the coefficients $p_{\sigma l}^{\mu m}$ have the same properties as the $b_{\sigma l}^{\mu m}.$

Thus, using \eqref{flat2} we conclude that
\begin{align*}\Delta(x^\mu r^{m} U) & =  \frac{U_0}{r}[m(m+1+2\mu_n) \, x^\mu r^{m-1} + r^m \mu_n x^{\mu - \bar n}\\
&+\mu_i(\mu_i-1)\,x^{\mu-2\bar i} r^{m+1}  +  c_{\sigma l}^{\mu m} x^\sigma r^l + \delta O(|X|^{k+\alpha})]\end{align*}
with $$c_{\sigma l}^{\mu m} \ne 0 \quad \mbox{only if} \quad \quad |\mu| + m-1<|\sigma|+l \le k,$$
and \begin{equation}\label{small}|c_{\sigma l}^{\mu m}|\le C \delta.\end{equation}

If $$P= a_{\mu m} x^\mu r^m \quad \mbox{is a polynomial of degree $k+1$,}$$
then
\begin{equation}\label{DUP}
\Delta (UP)= \frac{U_0}{r} (A_{\sigma l} x^\sigma r^l + \delta O(|X|^{k+\alpha})), \quad \quad \quad |\sigma|+l \leq k,
\end{equation}
with
\begin{align}\label{A} A_{\sigma l} & =  (l+1)(l+2 + 2\sigma_n) \, a_{\sigma,l+1}  + (\sigma_n +1) a_{\sigma +\bar n,l} +\\
 \nonumber &+ (\sigma_i +1)(\sigma_i+2) a_{\sigma+2\bar i,l-1} + c^{\mu m}_{\sigma l} a_{\mu m}.
 \end{align}
 
From \eqref{A} we see that $a_{\sigma, l+1}$ (whose coefficient is different than 0) can be expressed in terms of $A_{\sigma l}$ and a linear combination of $a_{\mu m}$ with $ \mu + m < |\sigma| + l +1$
plus a linear combination of $a_{\mu m}$ with $ \mu + m = |\sigma| + l +1$ and $m < l+1$. This shows that the coefficients $a_{\mu m}$ are uniquely determined from the linear system \eqref{A} once $A_{\sigma l}$ and $a_{\mu 0}$ are given.

 \begin{defn}\label{d1} We say that $P$ is approximating for $u/U$ at 0  if $A_{\sigma l}$ coincide with the coefficients of $R$.
 \end{defn}

{\it Remark:} Here we point out the difference between approximating using $U$ and $U_0$. If we want to obtain an expansion as in \eqref{DUP} for $\triangle (U_0P)$ then we need to require $\Gamma \in C^{k+2,\alpha}$ in order to deal with the terms $x^\mu r^m\triangle U_0$.

\smallskip

The following improvement of flatness lemma is the key ingredient in the proof of Proposition \ref{PSchauder}.

\begin{lem}\label{Imp} There exist universal constants $\rho$, $\delta$ depending only on $k, \alpha$ and $n$, such that if $P$ with $\|P\| \le 1$ is an approximating polynomial for $u/U$ in $B_\lambda$, that is $P$ is approximating for $u/U$ at 0 and $$\|u-U P\|_{L^\infty(B_\lambda)} \leq \lambda^{k+3/2+\alpha},$$ for some $\lambda>0$, then there exists an approximating polynomial $\bar P$ for $u/U$ at $0$ such that in $B_{\rho \lambda}$:
$$\|u -U\bar P\| _{L^\infty(B_{\rho\lambda})} \leq (\rho\lambda)^{k+3/2+\alpha}, \quad \quad \|\bar P- P\|_{L^\infty(B_\lambda)} \le C \lambda^{k+1+\alpha}.$$
\end{lem}
\begin{proof}
Set
 $$u-U P=:\lambda^{k+\frac 3 2 +\alpha} \tilde u (\frac X \lambda).$$
Thus, since $P$ is approximating
\begin{equation}\label{Ft}
 \Delta \tilde u(\frac X \lambda)= \lambda^{\frac 12 -k-\alpha}\left (F(X)-\delta \frac{U_0}{r}O(|X|^{k+\alpha})\right)=:\tilde F(\frac X \lambda).
 \end{equation}

Using the hypothesis on $u$ and $F$ we find
$$ |\tilde u(X)| \le 1, \quad \quad \quad |\Delta \tilde u(X)| \leq C \delta r^{-\frac 12}  \quad \text{in $B_1$.}$$
Denote by $\tilde \Gamma$, $\tilde {\mathcal P}$, $\tilde U_0$, $\tilde U$ the rescalings of $\Gamma$, $\mathcal P$, $U_0$ and $\tilde U$ from $B_\lambda$ to $B_1$ i.e. $$\tilde \Gamma:=\frac{1}{\lambda} \Gamma, \quad \tilde{\mathcal P}: = \frac 1 \lambda \mathcal P, \quad  \tilde U_0(X) := \lambda^{-\frac 12} U_0(\lambda X),  \quad \tilde U(X) := \lambda^{-\frac 12} U(\lambda X).$$

We decompose $\tilde u$ as $$\tilde u= \tilde u_0 + \tilde v$$ with
$$\begin{cases}
\Delta \tilde u_0 =0 \quad \quad\text{in $B_1 \setminus \tilde{\mathcal P}$},\\
\tilde u_0 = \tilde u \quad \quad \quad \text {on $\p B_1 \cup \tilde{\mathcal P}$,}
\end{cases}$$
and
$$\begin{cases}
|\Delta \tilde v| \leq C \delta r^{-\frac 12}   \quad \quad \text{in $B_1 \setminus \tilde{\mathcal P}$},\\
\tilde v = 0 \quad \quad \quad \quad \quad \text {on $\p B_1 \cup \tilde{\mathcal P}.$}\\
\end{cases}$$

Using barriers it follows that (see (5.6) in  \cite{DS3} or Lemma \ref{bar} in Section 5)
\begin{equation}\label{claim}
\|\tilde v\|_{L^\infty(B_1)} \leq C\delta \tilde U_0.
\end{equation}

To estimate $\tilde u_0$ we observe that  $\tilde u_0$ is a harmonic function in $B_1 \setminus \tilde{\mathcal P},$ $|\tilde u_0| \leq 1$ and as $\delta \rightarrow 0$, $\tilde {\Gamma}$ converges in the $C^{k+1,\alpha}$ norm to the hyperplane $\{x_n=0\}.$ 
Moreover, $\tilde u_0$ is uniformly H\"older continuous in $B_{1/2}$. 
By compactness, if $\delta$ is sufficiently small, $\tilde u_0$ can be approximated in $B_{1/2}$ by a solution of the Laplace problem with $\Gamma=\{x_n=0\}$. 
Thus by Theorem \ref{SchauderL}, and the fact that $\tilde U \to \tilde U_0$ uniformly as $\delta \to 0$ (see \eqref{flat2}) we deduce that 
\begin{equation}\label{Q}
\|\tilde u_0 - \tilde U Q\|_{L^\infty(B_\rho)} \leq C \rho^{k+2+\frac 1 2 }, \quad \quad \deg Q=k+1,
\end{equation}
with $\|Q\| \le C$. Since $U_0Q$ is harmonic we also get that the coefficients of $Q$
 satisfy (see \eqref{A})
\begin{equation}\label{fc}(l+1)(l+2 + 2\sigma_n) \,  q _{\sigma,l+1}  + (\sigma_n +1) \, q _{\sigma +\bar n,l} + (\sigma_i +1)(\sigma_i+2) \, q_{\sigma+2\bar i,l-1} =0,\end{equation}
with bounded $q_{\mu m}$.

Using also \eqref{claim} we find
$$\|\tilde u - \tilde U Q\|_{L^\infty(B_\rho)} \leq C \rho^{k+\frac 5 2 } +C\delta \leq \frac 1 2 \rho^{k+\frac 3 2 +\alpha}$$
provided that we choose first $\rho$ and then $\delta$, universal, sufficiently small.

Writing this inequality in terms of the original function $u$ we find,
$$\left |u-U \left (P+\lambda^{k+1+\alpha}Q(\frac X \lambda)\right ) \right | \leq \frac 1 2 (\lambda \rho)^{\frac 3 2 + \alpha} \quad \quad \mbox{in $B_{\rho\lambda}$}.$$
However $P(X)+\lambda^{k+1+\alpha}Q(X/ \lambda)$ is not an approximating polynomial and therefore we need to perturb $Q$ by a small amount.

Precisely,  we need to modify $Q$ into $\bar Q$ such that $\bar Q(x/\lambda, r/\lambda)$ is approximating for $R \equiv 0$. Thus its coefficients solve the system \eqref{A} with $A_{\sigma l}=0$ and rescaled $c^{\mu m}_{\sigma l}$, i.e.
\begin{align}\label {barQ}   (l+1)(l+2 + 2\sigma_n) \, \bar q_{\sigma,l+1}  + (\sigma_n +1) \bar q_{\sigma +\bar n,l} +&  \\
 \nonumber + (\sigma_i +1)(\sigma_i+2) \bar q_{\sigma+2\bar i,l-1} +  \bar c^{\mu m}_{\sigma l} \bar q _{\mu m}&=0,
 \end{align}
 with
 $$\bar c^{\mu m}_{\sigma l}:=\lambda^{|\sigma|+l+1-|\mu|-m}c^{\mu m}_{\sigma l}, \quad \quad \mbox{hence} \quad |\bar c^{\mu m}_{\sigma l}|\le |c^{\mu m}_{\sigma l}| \le C \delta.$$
 After subtracting \eqref{barQ} from \eqref{fc} we see that the coefficients  of $Q-\bar Q$ solve the linear system \eqref{barQ} with right hand side
 $A_{\sigma l} = \bar c^{\mu m}_{\sigma l}  q _{\mu m} $, hence $|A_{\sigma l}| \le C \delta$. As we mentioned before Definition 3.4, this system is uniquely solvable after choosing 
 $\bar q_{\mu 0}-q_{\mu 0}=0$ and we find
 $$\|\bar Q-Q\| \leq C\delta.$$
 This concludes the proof.
  \end{proof}

\begin{rem}
The classical boundary Harnack inequality implies that $|\tilde u_0| \le C \tilde U_0$ which together with \eqref{claim} gives
$$|\tilde u| \le C \tilde U_0 \quad \mbox{in $B_{1/2}$}.$$
This shows that the hypothesis of Proposition \ref{PSchauder} can be improved to
$$|u- UP| \le C U_0 \lambda^{k+1+\alpha} \quad \mbox{in $B_{\lambda/2}$}.$$

\end{rem}

We can now conclude the proof of Proposition \ref{PSchauder}. 

After multiplying $u$ by a small constant, we see that the hypotheses of the lemma are satisfied for some initial $\lambda_0$ small. Indeed, the coefficients of $R$ become sufficiently small and, by \eqref{A}, we can choose an initial approximating polynomial $P_{\lambda_0}$ with $\|P_{\lambda_0}\| \le 1/2$. Now we may iterate the lemma for all $\lambda=\lambda_0 \rho^m$ and conclude that there exists a limiting approximating polynomial $P_0$, $\|P_0\|\le 1$, such that $$|u-U P_0|\leq C|X|^{k+\frac 3 2 +\alpha} \quad \quad \mbox{in $B_1$.}$$ Moreover in view of the remark above the right hand side can be replaced by $C U_0 |X|^{k+1+\alpha}$ or equivalently by $C U |X|^{k+1+\alpha}$, and the proposition is proved.

\section{$C^{1,\alpha}$ boundaries}

We start by proving the Schauder estimate Theorem \ref{S_00}. For the reader's convenience we state it again.

\begin{thm}\label{S_0} Let $u$ be a solution to 
$$|\Delta u| \leq r^{\alpha - \frac 3 2} \quad \text{in $B_1 \setminus \mathcal P$}, \quad u=0 \quad \text{on $\mathcal P$}$$ with $\|u\|_{L^\infty} \leq 1$ and $\|\Gamma\|_{C^{1,\alpha}} \leq 1$. Then,
$$\left \|\frac{u}{U_0} \right \|_{C^\alpha_{xr}(\Gamma \cap B_{1/2})} \leq C$$ with $C>0$ depending on $n$ and $\alpha.$
\end{thm}

At the origin, the theorem states that there exists a constant $a$, $|a| \le C$ such that
\begin{equation}\label{30}
|u-aU_0| \le C |X|^\alpha U_0.
\end{equation}
It turns out that if $u$ is harmonic then we can differentiate formally the inequality above.

\begin{lem}\label{s_0}
Assume that $u$ is harmonic, satisfies the hypotheses of Theorem $\ref{S_0}$, and the expansion \eqref{30} holds. Then, for a.e. $X \in B_{1/2}$ we have
$$|\nabla u-  \nabla (a U_0)| \le C |X|^{\alpha} r^{- \frac 1 2} \quad \quad |\nabla _x u -\nabla_x (a U_0)| \le C |X|^{\alpha} \frac{ U_0}{r}.    $$
\end{lem}

We prove the Schauder estimates and boundary Harnack in slit domains with $C^{1,\alpha}$ boundary using the same strategy as in the case of $C^{k+1,\alpha}$ domains with $k \geq 1$. However,
due to the lack of regularity of $r, d$ and $U_0$ the ``test" functions in the proof for $k \geq 1$ must be slightly modified. We achieve this by working with ``regularizations" of the functions $r$, $U_0$ that we denote by $\bar r$, $\bar U_0$. Their main properties  are given in the next lemma. Notice that $r$, $U_0$ are differentiable a.e.

\begin{lem} \label{reg_1}Let $\|\Gamma\|_{C^{1,\alpha}} \leq \delta.$ There exist smooth functions $\bar r, \bar U_0$ such that 
$$\left | \frac {\bar r}{r}-1 \right |  \leq C \delta \, r^{\alpha}, \quad \quad \left |\frac{ \bar U_0}{U_0}-1 \right | \leq C \delta \, r^{\alpha},$$
$$ |\nabla \bar r -\nabla r|   \leq C\delta \, r^{\alpha}, \quad \quad |\p_{x_{n+1}}\bar r -\p_{x_{n+1}}r| \le C \delta \, r^{\alpha-\frac 12} U_0 ,$$
$$ |\Delta \bar r - \frac{1}{ r}| \leq C \delta \, r^{\alpha -1}, \quad \quad  |\Delta \bar U_0| \leq C\delta \, r^{\alpha- \frac 3 2},$$
with $C$ universal.
\end{lem}

The proof of Lemma \ref{reg_1} is postponed till Section \ref{T}. 

As usually, Theorem \ref{S_0} follows from the next improvement of flatness lemma. 
\

\begin{lem}\label{imp_0} Let
$$|\Delta u| \leq \delta r^{\alpha -3/2} \quad \text{in $B_1 \setminus \mathcal P$}, \quad u=0 \quad \text{on $\mathcal P$}, \quad \|\Gamma\|_{C^{1,\alpha}} \leq \delta.$$ Assume that there exists a constant $a$, $|a| \leq 1$ such that for some $\lambda>0$
$$\|u - a U_0\|_{L^\infty(B_\lambda)} \leq \lambda^{\alpha+1/2}.$$ Then there exists $\rho>0$ such that 
$$\|u - b U_0\|_{L^\infty(B_{\rho \lambda})} \leq (\rho \lambda)^{\alpha+1/2},$$ with $|a-b| \leq C \lambda^\alpha$, as long as $\delta$ is sufficiently small. 
\end{lem}

\begin{proof}From Lemma \ref{reg_1} we can replace $U_0$ by $\bar U_0$ and assume that 
$$|u - a \bar U_0| \leq 2 \lambda^{\alpha+1/2} \quad \text{in $B_\lambda$.}$$

Set $$u - a \bar U_0 = 2 \lambda^{\alpha+1/2} \tilde u(\frac{X}{\lambda}).$$ Then, using the bound for $\triangle  \bar U_0$, we obtain that
\begin{equation*} |\tilde u| \leq 1, \quad |\Delta \tilde u | \leq  C \delta r^{\alpha-3/2} \quad \text{in $B_1.$}
\end{equation*}

We now write,
$$\tilde u = \tilde u_1 + \tilde u_2,$$ with
$$\Delta \tilde u_1= \Delta \tilde u \quad \text{in $B_1 \setminus \mathcal P$}, \quad \tilde u_1 = 0 \quad \text{on $\p B_1 \cup \tilde{\mathcal P}$}$$
and
$$\Delta \tilde u_2 = 0  \quad \text{in $B_1 \setminus \mathcal P$}, \quad \tilde u_2 =\tilde u \quad \text{on $\p B_1 \cup \tilde{\mathcal P}.$}$$

By Lemma \ref{bar} in Section \ref{T} we have $\|\tilde u_1\|_{L^\infty} \le C \delta \tilde U_0$ hence $\tilde u_1$ converges to 0 uniformly as $\delta \to 0.$

To estimate $\tilde u_2$  we argue by compactness, as in the case $k \geq 1$. If $\delta$ is sufficiently small universal, $\tilde u_2$ can be approximated in $B_{1/2}$ by a solution of the Laplace problem with $\Gamma=\{x_n=0\}$. Thus by Theorem \ref{SchauderL}, we deduce that 
\begin{equation*}
\|\tilde u_2 - b \tilde U_0 \|_{L^\infty(B_\rho)} \leq C \rho^{1+\frac 1 2 }, 
\end{equation*} for some constant $b$, $|b| \leq C.$
Thus,
$$\|\tilde u - b \tilde U_0 \|_{L^\infty(B_\rho)} \leq C \rho^{1+\frac 1 2 } + o(\delta) \leq \frac 1 4 \rho^{\frac 1 2 +\alpha},$$
provided we choose first $\rho$ then $\delta$ sufficiently small.

Writing this inequality in terms of the original function $u$ we obtain (for $\delta$ small enough),
$$|u - a \bar U_0 - 2 b \lambda^\alpha  U_0| \leq \frac{1}{2} (\lambda \rho)^{\alpha +1/2} \quad \text{in $B_{\rho \lambda}$}.$$
Then, by Lemma \ref{reg_1}, we obtain
$$|u - (a+ 2 b  \lambda^\alpha) U_0| \leq (\lambda \rho)^{\alpha +1/2}, \quad\text{in $B_{\rho\lambda}$} $$ as desired.
\end{proof}

\smallskip

For the remaining of this section we prove our main Theorem \ref{SchauderO}, for $k=0$. For clarity of exposition we write below the statements of Proposition \ref{PSchauder} and Lemma \ref{Imp} leading to it, in the case $k=0$. 

\begin{prop}\label{PSchauder_0} Let $\Gamma \in C^{1, \alpha}$, $\|\Gamma\|_{C^{1,\alpha}} \le 1$, and $U$ be as in \eqref{FBU}.
Assume that $u\in C(B_1)$ is even and vanishes on $\mathcal P$, $\|u\|_{L^\infty} \le 1$, and
\begin{equation*}
\Delta u (X)= q \frac{U_0}{r} +F(X) \quad \quad \quad \quad \mbox{in $B_1 \setminus \mathcal P$},
\end{equation*}
with
$$|F(X)| \le r^{-\frac 12}|X|^{\alpha} \quad \quad \mbox{and $|q|\leq 1$.}$$
 There exists a polynomial $P(x,r)$ of degree $1$ with coefficients bounded by $C$ such that
 $$\left|\frac{u}{U}-P\right| \le C |X|^{1+\alpha,}$$  with $C$ depending on $\alpha$, $n$.
\end{prop}

As usually, after performing an initial dilation we may assume that:
\begin{equation}\label{flat1_0}\|g\|_{C^{1,\alpha}} \leq \delta, \quad |q| \leq \delta, \quad |F| \leq \delta r^{-1/2} |X|^{\alpha}.\end{equation}
Proposition \ref{PSchauder_0} will follow as in the case $k \ge 1$, after we extend the definition of approximating polynomial to this case i.e. now $P$ is a polynomial in $(x, \bar r),$ rather than $(x,r)$. 

Let $P(x,\bar r)$ be a polynomial of degree one in $x$ and $\bar r$,
$$P(x, \bar r) = a_0 + \sum_{i=1}^n a_i x_i + a_{n+1} \bar r.$$

We claim that
\begin{equation} \label{P} \Delta (U P) = \frac{U_0}{r} [a_n+2 a_{n+1} + O(\delta |X|^\alpha)].
\end{equation}

\

\begin{defn}We say that $P$ is approximating for $u/U$ at $0$ if 
$$ a_n+2 a_{n+1} = q. $$\end{defn}

The proof of the claim is postponed till later. Now, with this definition the proof of Proposition \ref{PSchauder_0} is a consequence of the next lemma whose proof is identical to the case $k \geq 1$.

\begin{lem}\label{Imp_0} There exist universal constants $\rho$, $\delta$ depending only on $\alpha$ and $n$, such that if $P_0(x,\bar r)$ with $\|P_0\| \le 1$ is an approximating polynomial for $u/U$ in $B_\lambda$, that is $P$ is approximating for $u/U$ at 0 and $$|u-U P|_{L^\infty(B_\lambda)} \leq \lambda^{3/2+\alpha},$$ for some $\lambda>0$, then there exists an approximating polynomial $ P_1(x,\bar r)$ for $u/U$ at $0$ such that in $B_{\rho \lambda}$:
$$|u -U P_1| _{L^\infty(B_{\rho\lambda})} \leq (\rho\lambda)^{3/2+\alpha}, \quad \quad \| P_1- P_0\|_{L^\infty(B_\lambda)} \le C \lambda^{1+\alpha}.$$
\end{lem}

Notice that in view of Lemma \ref{reg_1} $$|P(x,r)-P(x,\bar r)| \le C \delta \|P\| r^{1+\alpha},$$
and Proposition \ref{PSchauder_0} follows as before.

It remains to prove formula \eqref{P}. We compute that
\begin{equation}\label{UP} \Delta(UP)= a_{n+1} U \Delta \bar r + 2 a_i U_i + 2 a_{n+1} \nabla {\bar r} \cdot \nabla U. 
\end{equation}
Now we use Theorem \ref{S_0}, Lemma \ref{s_0} and estimate $U$, $\nabla U$ in terms of $U_0$, $\nabla U_0$ together with the estimates for $\nabla \bar r$, $\triangle \bar r$ from Lemma \ref{reg_1}. From Theorem \ref{S_0}, Lemma \ref{s_0}, we see that after multiplication by a constant and a dilation we may suppose that the function $U$ satisfies:
$$ U= U_0(1+ O(\delta |X|^\alpha)),$$ 
$$ \nabla_x U = \nabla_x{U_0} + O(\delta \frac{U_0}{r}|X|^\alpha),  \quad \quad \quad \p_{x_ {n+1}} U = \p_{x_{n+1}} U_0 + O(\delta r^{- \frac12}|X|^\alpha).$$
Since
$$\triangle \bar r=\frac 1 r + O(\delta r^{\alpha -1}),$$
$$\nabla_x \bar r=\nabla_x r + O(\delta r^\alpha), \quad \p_{x_ {n+1}} \bar r = \p_{x_{n+1}} r + O(\delta U_0 r^{\alpha- \frac 1 2}),$$
and also
$$\nabla_x U_0= \frac{U_0}{2r} \, \nabla_x d, \quad \quad \nabla_x d= e_n + O(\delta |X|^\alpha),$$
$$ \quad |\nabla U_0| \le C r^{- \frac 1 2}, \quad \quad \nabla r \cdot \nabla U_0=\frac{U_0}{2r}, \quad \quad |\p_{x_{n+1}} r| \le C U_0 r^{- \frac 1 2}, $$
we easily obtain \eqref{P} from \eqref{UP}.

\section{Proof of some technical lemmas}\label{T}

In this section we collect the proofs of several technical lemmas. We start with the approximation Lemma \ref{reg_1}. For the reader's convenience we state it again.

\begin{lem} \label{reg_0}Let $\|\Gamma\|_{C^{1,\alpha}} \leq 1.$ There exist smooth functions $\bar r, \bar U_0$ such that 
$$\left | \frac {\bar r}{r}-1 \right |  \leq C r^{\alpha}, \quad \quad \left |\frac{ \bar U_0}{U_0}-1 \right | \leq C  r^{\alpha},$$
$$ |\nabla \bar r -\nabla r|   \leq C r^{\alpha}, \quad \quad |\p_{x_{n+1}}\bar r -\p_{x_{n+1}}r| \le C  r^{\alpha-\frac 12} U_0 ,$$
$$ |\Delta \bar r - \frac{1}{ r}| \leq C  r^{\alpha -1}, \quad \quad  |\Delta \bar U_0| \leq C r^{\alpha-3/2},$$
with $C$ universal. Moreover if we assume $\|\Gamma\|_{C^{1,\alpha}} \leq \delta,$ small, then the constant $C$ above is replaced by $ C \delta$.
\end{lem}

\begin{proof} The idea is to smooth out the signed distance function $d$ to $\Gamma$ and then smooth out $r$ and $U_0$ using the formulas
$$r=\sqrt{d^2+x_{n+1}^2}, \quad \quad U_0=\frac{1}{\sqrt 2} \sqrt{r+d}.$$ 
We divide the proof in three steps. Whenever we write $\nabla d$, $\nabla r$, $\nabla U_0$ we assume we are at a point where these locally Lipschitz functions are differentiable.

\smallskip 

{\it Step 1.}
We start by constructing $\bar d$ by smoothing the signed distance $d$ to $\Gamma$ in dyadic tubular neighborhoods, and then we glue them together. 
First, define the open tubular neighborhood of $\Gamma,$
$$D_\lambda = \{x \in \R^n : |d|<\lambda\}, \quad \quad \mbox{$\lambda$ small.}$$
We set
$$d_\lambda := d * \rho_\lambda, \quad \rho_\lambda = \lambda^{-n}\rho(\frac x \lambda),$$ with $\rho$ a symmetric kernel supported in $B_{1/10}.$ 

We claim that 
\begin{equation} \label{dlambda} |d_\lambda - d| \leq C \lambda^{1+\alpha}, \quad |\nabla d_\lambda - \nabla d| \leq C  \lambda^\alpha, \quad |D^2 d_\lambda| \leq C \lambda^{\alpha -1} \quad \text{in $D_{4\lambda}.$}
\end{equation}
We check our claim at a point $x_0$ on the $x_n$-axis. Since $\|\Gamma\|_{C^{1,\alpha}} \leq 1$, we have
$$|d - x_n| \leq C \lambda^{1+\alpha}, \quad \text{in $B_{4\lambda}$,}$$
and we remark that if $\|\Gamma\|_{C^{1,\alpha}} \leq \delta$ then we may replace $C$ by $C \delta$.

Thus,
$$d=x_n +   \lambda^{1+\alpha} v, \quad |v| \leq C,$$
and using that $x_n*\rho_\lambda=x_n$ we get
$$d_\lambda = d *\rho_\lambda = x_n +  \lambda^{1+\alpha} \, \, (  v * \rho_\lambda).$$
This gives,
$$\nabla d_\lambda = e_n +  \lambda^{1+\alpha}\, ( v * \nabla \rho_\lambda), \quad \quad D^2d_\lambda =  \lambda^{1+\alpha} \, (v * D^2 \rho_\lambda)$$
and the claim \eqref{dlambda} follows by using that
$$\int \lambda |\nabla \rho_\lambda| dx \leq C, \quad \int \lambda^2 |D^2 \rho_\lambda| dx \leq C, \quad |\nabla d(x_0) - e_n| \le C \lambda^\alpha.$$

The function $d_\lambda$ approximates $d$ up to an error $\lambda^{1+\alpha}$ in $D_\lambda$. Next we interpolate between various $d_\lambda$ with $\lambda=\lambda_k=4^{-k}$ in the annular sets $\mathcal A_\lambda := \{\lambda < d< 4\lambda\}.$

We define $\bar d$ to coincide with $d_\lambda$ in $\mathcal A_\lambda \cap D_{2\lambda}$ and with $d_{4\lambda}$ in $\mathcal A_\lambda \setminus D_{3\lambda}$. Precisely let
$$\bar d = \varphi  \, d_\lambda+ (1-\varphi) \, d_{4\lambda}$$
with $\varphi$ a cutoff function
\begin{equation}\label{phi}\varphi=0 \quad \text{if $d>3\lambda$}, \quad \varphi =1 \quad \text{if $d < 2\lambda$.}\end{equation}
We set 
$$\varphi=h \left (\frac{d_\lambda}{\lambda} \right),$$
with $$h(t)=1 \quad \text{if $t \leq 2+1/4$}, \quad h(t)=0 \quad \text{if $t \geq 2+3/4$}$$
and $h$ smooth in between. Thus, 
\begin{equation}\label{cutoff}|\nabla \varphi| \leq C\lambda^{-1}, \quad |D^2\varphi| \leq C \lambda^{-2}.\end{equation}
Then, $\bar d$ satisfies
\begin{equation}\label{35} |\bar d - d| \leq C \lambda^{1+\alpha}, \quad |\nabla \bar d - \nabla d| \leq C \lambda^\alpha, \quad |D^2 \bar d| \leq C \lambda^{\alpha -1} \quad \text{in $\mathcal A_\lambda.$}
\end{equation}
This follows immediately after computing 
\begin{align}\label{dbar1}
 \nabla \bar d & =  \varphi \, \nabla d_\lambda + (1-\varphi) \nabla d_{4\lambda}  + (d_\lambda - d_{4\lambda}) \nabla \varphi,\\
\nonumber D^2 \bar d & = \varphi D^2 d_\lambda +(1-\varphi) D^2 d_{4\lambda} + 2 (\nabla d_\lambda - \nabla d_{4\lambda}) \otimes \nabla \varphi + (d_\lambda - d_{4\lambda}) D^2 \varphi,
\end{align}
and then using \eqref{dlambda} and \eqref{cutoff}.


\

{\it Step 2.} We construct $\bar r$ in a similar fashion as $\bar d$. We first construct approximations $r_\lambda$ in dyadic annular regions $\mathcal R_\lambda$ in $\R^{n+1}$, and then we ``glue" them together. 

We define
\begin{equation}\label{defrlambda}
r_\lambda:= \sqrt{d_\lambda^2 + x_{n+1}^2}, \quad \text{in $\mathcal R_\lambda= \{\frac{\lambda}{2} < r < 4\lambda\}.$}
\end{equation}
with $d_\lambda$ as in Step 1. Then $r$, $r_\lambda$ and $\lambda$ are all comparable to each other in $\mathcal R_\lambda$ and we claim that
\begin{equation}\label{rlambda} |r_\lambda - r| \leq C  \lambda^{1+\alpha}, \quad \quad |\Delta r_\lambda - \frac{1}{r}| \leq C  \lambda^{\alpha -1}
\end{equation}
\begin{equation}\label{333}
|\nabla r_\lambda - \nabla r| \leq C  \lambda^\alpha, \quad \quad \|D^2 r_\lambda\| \le \frac{C}{\lambda}   .
 \end{equation}
Indeed, using the first inequality in \eqref{dlambda} we obtain
$$|r_\lambda^2 - r^2| = |d_\lambda^2- d^2| \leq C \lambda^{2+\alpha}.$$
This gives the first inequality in \eqref{rlambda} hence
\begin{equation}\label{compr} |\frac{r_\lambda}{r} -1| \leq C \lambda^\alpha.\end{equation}
Since
\begin{equation}\label{nabr}
\nabla r_\lambda= \frac{1}{r_\lambda} (d_\lambda \nabla_xd_\lambda, x_{n+1})
\end{equation}
the inequalities in \eqref{333} follow easily from the estimates for $d_\lambda$ (see \eqref{dlambda}, \eqref{compr}). 

From \eqref{333}, \eqref{dlambda} we obtain
$$|\nabla r_\lambda|-1=O( \lambda^\alpha), \quad \quad |\nabla d_\lambda|-1=O( \lambda^\alpha).$$
Then the identity
$$ r_\lambda \Delta r_\lambda + |\nabla r_\lambda|^2=\frac 1 2 \Delta r_\lambda^2= \frac 12 \Delta (d_\lambda^2+x_{n+1}^2)=d_\lambda \Delta d_\lambda + |\nabla d_\lambda|^2+1,$$
implies $$r_\lambda \triangle r_\lambda=1+O( \lambda^\alpha),$$
which gives the second inequality in \eqref{rlambda} and our claim is proved.

\smallskip

Next we glue various $r_\lambda$'s with $\lambda=\lambda_k=4^{-k}$. In the regions $\{\lambda_k < r < 4 \lambda_k  \}$ we define 
$$\bar r = \varphi \,  r_\lambda  + (1-\varphi)r_{4\lambda}, \quad \quad \mbox{with} \quad \varphi:=h( \frac {r_\lambda}{\lambda}) ,\quad \quad \mbox{$h$ as above.}$$
Notice that $\varphi$ satisfies \eqref{cutoff}, and \eqref{333} holds with $d$ replaced by $r$. 
This shows that $\bar r$ still satisfies in this region
\begin{equation} |\bar r - r| \leq C r^{1+\alpha}, \quad |\nabla \bar r - \nabla r| \leq C r^\alpha, \quad |\Delta  \bar r - \frac{1}{r}| \leq C r^{\alpha -1}.\end{equation}
Moreover $$\mbox{ \eqref{nabr} and} \quad |\p_{x_{n+1}} \varphi | \le C \frac{|x_{n+1}|}{ \lambda^2} \quad \Longrightarrow \quad |\p_{x_{n+1}} \bar r- \p_{x_{n+1}} r| \le C \frac{|x_{n+1}|}{r} \lambda^\alpha \le C \frac{U_0}{r^\frac 12} \lambda^\alpha.$$ 
\smallskip

{\it Step 3.} We construct $\bar U_0.$ As before,
$$(U_0)_\lambda = \frac{1}{\sqrt 2}(d_\lambda + r_\lambda)^{1/2}, \quad \text{in $\mathcal R_\lambda$.}$$
Below we show that $(U_0)_\lambda$ satisfies the following inequalities
\begin{equation}\label{U0lambda} 
\left | \frac{(U_0)_\lambda}{U_0} - 1 \right | \leq C \lambda^\alpha, \quad \quad |\nabla (U_0)_\lambda-\nabla U_0| \le C \lambda^{\alpha-\frac 1 2}, 
\end{equation}
\begin{equation}\label{U0lambda1} 
|\triangle (U_0)_\lambda| \le C \lambda^{\alpha -\frac 32}.
\end{equation}
Then the function $\bar U_0$ is obtained as in Step 2 by interpolating the various $(U_0)_\lambda$'s.

We check the inequalities above separately in the regions
$$\mathcal R_\lambda^1:=\mathcal R_\lambda \cap \{d > -r / 2  \}  \quad \mbox{and}  \mathcal  \quad \mathcal R_\lambda^2:=\mathcal R_\lambda \cap \{d < - r/  2  \},$$
depending whether or not we are closer to the set $\mathcal P$ where $U_0$ and $(U_0)_\lambda$ vanish.

{\it Case 1:} In the region $\mathcal R_\lambda^1$ we know that $U_0$, $(U_0)_\lambda$ and $\lambda^{1/2}$ are comparable to each other and
$$(U_0)_\lambda = U_0 \left (\frac{ r_\lambda + d_\lambda}{r+d}\right)^{\frac 12}.$$
From the Step 1 and 2 we have
$$\frac{ r_\lambda + d_\lambda}{r+d}=1+ O ( \lambda^{\alpha }), \quad \quad  \nabla \frac{ r_\lambda + d_\lambda}{r+d}=O ( \lambda^{\alpha -1}) $$
and we easily obtain the two inequalities in \eqref{U0lambda}. In particular we find
$$|\nabla (U_0)_\lambda|=|\nabla U_0| + O( \lambda ^{\alpha - \frac 12})=\frac {1}{2} r^{-1/2} + O( \lambda ^{\alpha - \frac 12}). $$
Now \eqref{U0lambda1} follows from the identity (see also \eqref{dlambda}, \eqref{rlambda}) 
$$(U_0)_\lambda \triangle (U_0)_\lambda +|\nabla (U_0)_\lambda|^2= \frac 14 \triangle (d_\lambda + r_\lambda).$$

{\it Case 2:} In the region $\mathcal R_\lambda^2$ we know that $U_0$, $(U_0)_\lambda$ and $|x_{n+1}|\lambda^{-1/2}$ are comparable to each other since
$$(U_0)_\lambda=|x_{n+1}|(r_\lambda - d_\lambda)^{-1/2} = U_0 \left (\frac{ r_\lambda - d_\lambda}{r-d}\right)^{-\frac 12},$$
and  \eqref{U0lambda} is obtained as above. In order to prove \eqref{U0lambda1} we write (assume $x_{n+1}>0$)
\begin{align}\label{long}
\nonumber \triangle (U_0)_\lambda &= 2 \p _{x_{n+1}}(r_\lambda-d_\lambda)^{-1/2} + x_{n+1} \triangle (r_\lambda-d_\lambda)^{-1/2}   \\
&=\frac{x_{n+1}}{2}(r_\lambda-d_\lambda)^{-\frac 32} \left(-2\frac{1}{r_\lambda} - \triangle (r_\lambda-d_\lambda) +\frac 3 2 \frac{ |\nabla (r_\lambda-d_\lambda)|^2}{r_\lambda-d_\lambda}          \right),
\end{align}
 and we used $\p_{x_{n+1}} r_\lambda=x_{n+1}/r_\lambda$ and $\p_{x_{n+1}} d_\lambda=0$. Since
\begin{align*}
|\nabla (r_\lambda-d_\lambda)|^2&= r_\lambda^{-2} \left((r_\lambda -d_\lambda)^2 |\nabla d_\lambda|^2 + x_{n+1}^2 \right)  \\
&=r_\lambda^{-2} \left( 2r_\lambda(r_\lambda-d_\lambda) +O( \lambda^{2+\alpha}) \right)   
\end{align*}
we find that the quantity in the parenthesis in \eqref{long} is $O( \lambda^{\alpha-1})$ and our claim is proved.

\end{proof}

In the next lemma we obtain an $L^\infty$ bound for solutions to the Laplace equation with right hand side that degenerates near $\Gamma$.

\begin{lem}[Barrier]\label{bar}
Assume $\|\Gamma\|_{C^{1,\alpha}} \leq \delta$ with $\alpha \in (0, \frac 12),$ and
$$|\Delta u| \leq  r^{\alpha -\frac 32} \quad \text{in $B_1 \setminus \mathcal P$}, \quad u=0 \quad \text{on $\mathcal P \cup \p B_1$}.$$
Then
$$|u| \le C U_0  \quad \quad \mbox{in $B_1$},$$
with $C$, $\delta$ depending on $n$ and $\alpha$.
\end{lem}

\begin{proof}
We construct an upper barrier for $u$. Let
$$V:= \bar U_0 - \bar U_0 ^{1+2 \alpha},$$
and notice that $V \ge 0$ in $B_1$.
We compute
$$\triangle V = \triangle \bar U_0 - (1+2 \alpha) \bar {U_0}^{2 \alpha-1} (\bar U_0 \triangle \bar U_0 + 2\alpha |\nabla \bar U_0|^2).$$
By Lemma \ref{reg_0} we have
$$ |\bar U_0|\le r^\frac 12, \quad \quad | \triangle \bar U_0| \le C  \delta r ^{\alpha -\frac 32}, \quad \quad |\nabla \bar U_0| \ge c r^{-\frac 12},$$
thus we obtain
$$\triangle V \le -c r^{\alpha -\frac 32}.$$
Since $V \ge 0$ and $\triangle (C V) \le \triangle u$, we apply maximum principle in $B_1 \setminus \mathcal P$ and obtain
$$ u \le C V \le C \bar U_0 \quad \mbox{in $B_1$.}$$
\end{proof}

\begin{rem}
$L^\infty$ bounds for $u$ hold also for more degenerate right hand side. If for some small $\gamma >0$, 
$$|\triangle u| \le r^{\gamma-2}, \quad u=0 \quad \text{on $\mathcal P \cup \p B_1$},$$
then $$|u| \le C r^\gamma.$$
Indeed, in this case we can use $V=\bar U_0^{2 \gamma}$ as an upper barrier. 
\end{rem}

\

{\it Proof of Lemma $\ref{s_0}$}.
We have $$|u-aU_0| \le  C r^{\alpha} U_0 \le C r^{\frac 12 + \alpha}.$$

The idea is to replace $U_0$ with an appropriate function $U_0^*$ which measures the distances $d$ and $r$ to a straight boundary instead of $\Gamma$. Notice that $u$ and $U_0^*$ are both harmonic.
  
We pick a point $X_0$ at distance $\lambda$ from $\Gamma$. Assume for simplicity of notation that the closest point to $X_0$ on $\Gamma$ is the origin $0$, thus $X_0$ belongs to the 2D plane $\{x'=0\}$.
Let $U_0^*$ denote the 1-dimensional solution with respect to the straight boundary $L:=\{x_n=0\}$, i.e.
$$U_0^*:=U_{0,L}=\frac{1}{\sqrt 2}\sqrt{x_n+r^*}, \quad \quad r^*:=\sqrt{x_n^2+x_{n+1}^2}.$$
 Notice that $U_0^*$, $r^*$ coincide with $U_0$, $r$ at the point $X_0$. Moreover, if $d$, $r$, $U_0$ are differentiable at $X_0$ then
 $$\nabla d=\nabla x_n=e_n, \quad  \nabla U_0=\nabla U_0^*, \quad \nabla r=\nabla r^* \quad \quad \mbox{at $X_0$} .$$

In the conical region $$\mathcal C:= \{  |x'| < r^* \} \cap \{ \frac{\lambda}{2} < r^* < 2 \lambda\}$$
we use that $\|\Gamma\|_{C^{1,\alpha}} \le 1$ and obtain (as in \eqref{U0lambda})
$$|U_0^*-U_0| \le C \lambda^{\frac 12 + \alpha}  $$
thus
$$|u-aU_0^*| \le C \lambda^{\frac 12 + \alpha}  \quad \mbox{in $\mathcal C$.}$$
Since $u-aU_0^*$ is harmonic and vanishes on $\mathcal P$ we apply gradient estimates and obtain
$$|\nabla u- a\nabla U_0^*| \le C \lambda^{\alpha - \frac 12}, \quad \quad |\nabla_x u-   a\nabla_x U_0^*| \le C \lambda^{\alpha - 1} U_0^* \quad \quad \mbox{at $X_0$,}  $$
and we replace $U_0^*$ by $U_0$ in the inequalities above.

In conclusion, at an arbitrary point $X \in B_{1/2}$ where $U_0$ is differentiable we have
$$|\nabla u- a_{\pi(X)}\nabla U_0| \le C r^{\alpha - \frac 12}, \quad \quad |\nabla_x u-   a_{\pi(X)}\nabla_x U_0| \le C r^{\alpha - 1} U_0,$$
where $\pi(X)$ is the projection of $X$ onto $\Gamma$ and $a_{\pi (X)}$ represents the corresponding constant for the expansion of $u$ at $\pi(X)$. The lemma is proved since
$$|a_{\pi(X)}-a| \le C |\pi(X)|^\alpha \le C |X|^\alpha, \quad \quad r \le |X|, $$
$$|\nabla U_0| \le C r^{-\frac 12}, \quad \quad \nabla_x U_0=\frac {U_0}{2r} \nabla_x d, \quad \quad |\nabla_x U_0| \le C \frac{U_0}{r}.$$

\qed

\

Finally we conclude with the proof of Lemma \ref{pr}.

\

{\it Proof of Lemma $\ref{pr}.$}
By Lemma 5.5 in \cite{DS3} we see that in the cone $\mathcal C_0=\{|(x_n,x_{n+1})|>|x'| \}$ we have
\begin{align*}
U_i&=\p_{x_i}(U_0 P_0) + O(\frac{U_0}{r}|X|^{k+\alpha})  \quad \quad i=1,2,..,n,\\
U_{n+1}&=\p_{x_{n+1}}(U_0 P_0) + O(|X|^{k-\frac 12+\alpha}).
\end{align*}
Since $|\p_{x_{n+1}}r| \le r^{-\frac 12} U_0$ we find
$$\p_r U=\p_r (U_0 P_0) + O(\frac{U_0}{r}|X|^{k+\alpha}) \quad \mbox{in $\mathcal C_0$,}$$
and the conclusion of the lemma follows as in \cite{DS3}, by writing the equality above for all corresponding cones $\mathcal C_Z$, $Z \in \Gamma$.

\qed

\end{document}